\documentclass[11pt]{article}
\usepackage{amsfonts,amsthm, amsmath}
\usepackage{epsfig}
\usepackage{makeidx}
\usepackage{graphicx,epstopdf}
\usepackage{float}
\usepackage{amsmath}
\usepackage{comment}
\usepackage{enumerate}
\DeclareGraphicsExtensions{.ps,.png,.jpg}

\makeindex
\usepackage{color}

\newcommand{\ncom}{\newcommand}
\ncom{\ul}{\underline}
\ncom{\beq}{\begin{equation}}
\ncom{\eeq}{\end{equation}}
\ncom{\bea}{\begin{eqnarray*}}
\ncom{\eea}{\end{eqnarray*}}
\ncom{\beqa}{\begin{eqnarray}}
\ncom{\eeqa}{\end{eqnarray}}
\ncom{\nno}{\nonumber}
\ncom{\non}{\nonumber}
\ncom{\ds}{\displaystyle}
\ncom{\half}{\frac{1}{2}}
\ncom{\mbx}{\makebox{.25cm}}
\ncom{\hs}{\mbox{\hspace{.25cm}}}
\ncom{\rar}{\rightarrow}
\ncom{\Rar}{\Rightarrow}
\ncom{\noin}{\noindent}
\ncom{\bc}{\begin{center}}
\ncom{\ec}{\end{center}}
\ncom{\sz}{\scriptsize}
\ncom{\rf}{\ref}
\ncom{\s}{\sqrt{2}}
\ncom{\sgm}{\sigma}
\ncom{\Sgm}{\Sigma}
\ncom{\psgm}{\sigma^{\prime}}
\ncom{\dt}{\delta}
\ncom{\Dt}{\Delta}
\ncom{\lmd}{\lambda}
\ncom{\Lmd}{\Lambda}
\ncom{\Th}{\Theta}
\ncom{\e}{\eta}
\ncom{\eps}{\epsilon}
\ncom{\pcc}{\stackrel{P}{>}}
\ncom{\lp}{\stackrel{L_{p}}{>}}
\ncom{\dist}{{\rm\,dist}}
\ncom{\sspan}{{\rm\,span}}
\ncom{\re}{{\rm Re\,}}
\ncom{\im}{{\rm Im\,}}
\ncom{\sgn}{{\rm sgn\,}}
\ncom{\ba}{\begin{array}}
\ncom{\ea}{\end{array}}
\ncom{\hone}{\mbox{\hspace{1em}}}
\ncom{\htwo}{\mbox{\hspace{2em}}}
\ncom{\hthree}{\mbox{\hspace{3em}}}
\ncom{\hfour}{\mbox{\hspace{4em}}}
\ncom{\vone}{\vskip 2ex}
\ncom{\vtwo}{\vskip 4ex}
\ncom{\vonee}{\vskip 1.5ex}
\ncom{\vthree}{\vskip 6ex}
\ncom{\vfour}{\vspace*{8ex}}
\ncom{\norm}{\|\;\;\|}
\ncom{\integ}[4]{\int_{#1}^{#2}\,{#3}\,d{#4}}
\ncom{\vspan}[1]{{{\rm\,span}\{ #1 \}}}
\ncom{\dm}[1]{ {\displaystyle{#1} } }
\ncom{\ri}[1]{{#1} \index{#1}}

\newtheorem{theorem}{\bf Theorem}[section]
\newtheorem{remark}{\bf Remark}[section]
\newtheorem{proposition}{Proposition}[section]

\newtheoremstyle
    {remarkstyle}
    {}
    {11pt}
    {}
    {}
    {\bfseries}
    {:}
    {     }
    {\thmname{#1} \thmnumber{#2} }

\theoremstyle{remarkstyle}

\begin{document}

\begin{center}
{\Large \bf Tempered Stable Autoregressive Models}
\end{center}
\vone
\vone
\begin{center}
{Niharika Bhootna}$^{\textrm{a}}$, {Arun Kumar}$^{\textrm{a*}}$
\footnotesize{
		$$\begin{tabular}{l}
		$^{\textrm{a}}$ \emph{Department of Mathematics, Indian Institute of Technology Ropar, Rupnagar, Punjab - 140001, India}\\
		$^{\textrm{*}}$ \emph{Correspondence: arun.kumar@iitrpr.ac.in}
\end{tabular}$$}
\end{center}
\vtwo
\begin{center}
\noindent{\bf Abstract}
\end{center}
In this article, we introduce and study a one sided tempered stable first order autoregressive model called TAR(1). Under the assumption of stationarity of the model, the marginal probability density function of the error term is found. It is shown that the distribution of the error term is infinitely divisible. Parameter estimation of the introduced TAR(1) process is done by adopting the conditional least square and method of moments based approach and the performance of the proposed methods are evaluated on simulated data.  Also we study an autoregressive model of order one with tempered stable innovations. Using appropriate test statistic it is shown that the model fit very well on real and simulated data. Our models generalize the inverse Gaussian and one-sided stable autoregressive models existing in the literature.

\vtwo

\noindent{\it Key Words:} Positive tempered stable distribution, autoregressive time-series model, complex inversion of Laplace transforms, estimation, infinite divisibility.
\vone


\vone

\section{Introduction}
Autoregressive models serve as eminent approach of modelling among different time-series methods. Many real life time-series data exhibits autoregressive behaviour. Classical autoregressive time series models are based on the assumption of normality of the innovation term. However, this assumption may not be true for all the real life scenarios. Many real life time series probability distributions display  heavy-tailed behaviour or semi-heavy tailed behaviour. Semi-heavy tailed pdf are those where tails are heavier than the Gaussian and lighter than the power law, see e.g. Omeya et al. \cite{Omeya2018}. It is well known that series of counts, proportions, binary outcomes or non-negative observations are some examples of non-normal real life time-series data see e.g. Grunwald et al. \cite{Grunwald1996}. The autoregressive model of order 1 denoted by AR(1) is a simple, useful and interpretable model in a wide range of real life applications.
Several AR(1) models with different marginal distributions are considered in literature. For example, Gaver and Lewis \cite{Gaver1980} considered linear additive first-order autoregressive scheme with gamma distributed marginals. The authors show that error distribution is same as the distribution of a non-negative random variable which has a point mass at 0 and which is exponential if positive (see Gaver and Lewis \cite{Gaver1980}).  A more general exponential ARMA$(p,q)$ model called EARMA$(p,q)$ is considered in Lawrance and Lewis \cite{Lawrance1980}. A second order autoregressive model is considered in Dewald and Lewis \cite{Dewald1985}, where Laplace distributed marginals are considered. Abraham and Balakrishna \cite{Abraham1999} considered inverse Gaussian autoregressive model where the marginals are inverse Gaussian distributed. The authors shows that after fixing a parameter of the inverse Gaussian to 0 the error term is also inverse Gaussian distributed. 

\noindent In this paper, we consider tempered stable autoregressive model of order 1 and denote it by TAR(1). This model generalizes the work discussed in Abraham and Balakrishna \cite{Abraham1999}. Since  one parameter inverse Gaussian (also called L\'evy distribution) is a particular case of the one sided tempered stable distribution. We derive the explicit form of the error density. Further, we show that if the AR(1) series marginals are stable distributed the error is also stable distributed with some scaling. A step-by-step procedure of the estimation of the parameters of the proposed model is given. Moreover, we consider an AR(1) model with one sided tempered stable innovations. This model is different from the TAR(1) model where marginals are one sided tempered stable and here the innovations are positive tempered stable distributed. The application of this model is shown on real life power consumption data. The rest of the paper is organized as follows: Section 2 defines the TAR(1) model where the explicit form of the density of the error term , infinite divisibility, moments and parameter estimation methods of TAR(1) model are discussed. In Section 3, an AR(1) model with positive tempered stable innovations is introduced. The estimation procedure of the parameters of the introduced model is given based on method of moments and conditional least square.  The simulations study, where efficacy of the estimation procedure is discussed based on simulated data and real life application are discussed in this section. The last section concludes.

\section{Tempered Stable Autoregressive Model}
In this section, we introduce the tempered stable autoregressive model and discuss the main properties. Consider an autoregressive process of order $1$, defined as:
\begin{equation}\label{main}
X_n=\rho X_{n-1}+\epsilon_n,
\end{equation}
where $|\rho|<1$ and $\{\epsilon_n, n\geq 1\}$ is sequence of i.i.d. random variables. Assuming marginals $X_n\textsc{\char13}s$ to be stationary positive tempered stable distributed. Then there exist a distribution of $\epsilon_n\textsc{\char13}s$. The class of stable distributions is denoted by $Stable(\beta,\nu,\mu,\sigma)$, with parameter $\beta \in(0,2]$ is the stability index, $\nu \in[-1,1]$ is the skewness parameter, $\mu\in\mathbb{R}$ is the location parameter and $\sigma>0$ is the shape parameter. The stable class probability density
functions do not possess closed form except for three cases (Gaussian, Cauchy, and
L\'evy). Generally stable distributions are represented in terms of their characteristic functions or Laplace transforms. Stable distributions are infinitely divisible. The one-sided stable random variable $S$ has following Laplace transform (see Samorodnitsky and Taqqu \cite{Samorodnitsky1994})
\begin{equation}
\mathbb{E}(e^{-s S}) = e^{- s^{\beta}},\;s>0,\;\beta\in(0,1).
\end{equation}
 The right tail of the $S$ behaves (see e.g. Samorodnitsky and Taqqu \cite{Samorodnitsky1994})
\begin{equation}\label{tail-stable}
\mathbb{P}(S>x) \sim \frac{ x^{-\beta}}{\Gamma{(1-\beta)}},\;\mbox{as}\;x\rightarrow\infty.
\end{equation}
Next, we introduce the positive tempered stable distribution. The positive tempered stable random variable $T$  with tempering parameter $\lambda>0$ and stability index $\beta \in(0,1)$, has the Laplace transform (LT) (see e.g. Meerschaert et al. \cite{Meerschaert2013}, Kumar and Vellaisamy \cite{Kumar2011})
\begin{equation}\label{tem-LT}
\mathbb{E}\left(e^{-s T}\right)=
e^{-\big((s + \lambda)^{\beta}-\lambda^{\beta}\big)}. \end{equation} 
Note that tempered stable distributions are obtained by exponential tempering in the distributions of $\alpha$-stable distributions  (see Rosinski \cite{Rosinski2007}). The advantage of tempered stable distribution over an $\alpha$-stable distribution is that it has moments of all order and its density is also infinitely divisible.  The probability density function for $T$ is given by 
\begin{equation}\label{ts-density}
 f_{\beta, \lambda}(x)= e^{-\lambda x+\lambda^{\beta}} f_{\beta}(x),~~ \lambda>0, \;\beta\in (0,1), 
\end{equation}
where $f_{\beta}(x)$ is the PDF of an $\alpha$-stable distribution (Uchaikin and Zolotarev \cite{Uchaikin1999}). 
The tail probability of tempered stable distribution has the following asymptotic behavior
\begin{align}\label{tail-TSS}
\mathbb{P}(T> x) &\sim c_{\beta,\lambda}\frac{e^{-\lambda x}}{x^{\beta}},\;\mbox{as}\;x\rightarrow \infty,
\end{align}
where $c_{\beta,\lambda} = \frac{1}{\beta\pi}\Gamma(1+\beta)\sin(\pi\beta)e^{\lambda^{\beta}}.$
The first two moments of tempered stable distribution are given by
\begin{equation}\label{moments-tss}
\mathbb{E}(T) = \beta \lambda^{\beta-1}, \;\; \mathbb{E}(T)^2 = \beta(1-\beta) \lambda^{\beta-2} + (\beta \lambda^{\beta-1})^2. 
\end{equation}
A pdf $u(x)$ is called a semi-heavy tailed pdf if
$$
u(x) \sim e^{-cx} v(x),\; c>0,
$$
where $v$ is a regularly varying function (see e.g. Omeya et al. \cite{Omeya2018}). Recall that a positive function $v$ is regularly varying with index $\nu$ if
$$
\lim_{x\rightarrow\infty}\frac{v(dx)}{v(x)} = d^{\nu}, \; d>0.
$$
Using \eqref{tail-TSS}, it is straight forward to show that
$$
f_{\beta, \lambda}(x) \sim \lambda c_{\beta, \lambda}x^{-\beta}e^{-\lambda x},\;\mbox{as}\; x\rightarrow \infty,
$$
and hence the tempered stable density function is semi-heavy tailed.

\subsection{Distributional Properties}
In this subsection, we discuss about the distributional properties related to the introduced TAR(1) process. If each $X_n$ in \eqref{main} is positive tempered stable then the Laplace transform of $X_n$ is given by
\begin{align}\label{LTAR1}
\Phi_{X_n}(s)&=\mathbb{E}(\exp(-s X_n))=\mathbb{E}(\exp(-s (\rho X_{n-1}+\epsilon_n))) = \Phi_{X_{n-1}}(s) \Phi_{\epsilon_n}(s).
\end{align}

\noindent Assuming $X_n\textsc{\char13}s$ to be stationary distributed, it follows

\begin{equation}\label{LT1}
\Phi_{X}(s)=\exp\{-((\lambda+s)^\beta-\lambda^\beta)\}, \hspace{4mm}\lambda>0,\hspace{1mm}\beta \in (0,1),
\end{equation}
and
\begin{equation}\label{LT2}
\Phi_{X}(\rho s)=\exp\{-((\lambda+\rho s)^\beta-\lambda^\beta)\}, \hspace{4mm}\lambda>0,\hspace{1mm}\beta \in (0,1).
\end{equation}

\noindent Using \eqref{LTAR1} 

\begin{equation}
\Phi_{X}(s)= \Phi_{\epsilon}(s)\Phi_{X}(\rho s),
\end{equation}
which implies
\begin{equation}
\Phi_{\epsilon}(s)= \dfrac{\Phi_{X}(s)}{\Phi_{X}(\rho s)}.
\end{equation}
Putting values from \eqref{LT1} and \eqref{LT2}, yields

\begin{equation} \label{der}
\Phi_{\epsilon}(s)= \exp\{(\rho s+\lambda)^\beta-(s+\lambda)^\beta\}.
\end{equation}
\begin{remark}
For $\lambda =0$ in \eqref{der}, we have $\Phi_{\epsilon}(s) = e^{-(1-\rho^{\beta})s^{\beta}}$, which is the Laplace transform of a positive stable random variable $(1-\rho^{\beta})^{1/\beta}S$. Hence, it shows that the if the AR(1) series is stable the error term is also stable.
\end{remark}
\begin{proposition}\label{2.1}
The error distribution with Laplace transform \eqref{der} is infinitely divisible.
\end{proposition}
\begin{proof}
By Feller \cite{Feller1971}, a pdf $f(x),\;x \geq 0$ is infinitely divisible iff its Laplace transform is of the form $F(s) = e^{-\phi(s)},\; s > 0$ where $\phi(s)$ has a completely monotone derivative. Further, a function $\psi(\cdot)$ is completely monotone if it possesses derivatives $\psi^{(n)}( \cdot )$ of all orders and $(-1)^n\psi^{(n)}(s)\geq 0,\;s > 0,\;n = 0, 1, \cdots $ (see e.g. Feller\cite{Feller1971}, p. 439). For $\Phi_{\epsilon}(s)$, we have $\phi(s) = (s+\lambda)^{\beta}-(\rho s+\lambda)^{\beta}$ which gives $\phi'(s) = \beta(s+\lambda)^{\beta-1}-\beta\rho(\rho s+\lambda)^{\beta-1}.$ Consider $\psi(s) = (s+\lambda)^{\beta-1}$, which yields to $\psi^{(n)}(s) = (\beta-1)(\beta-2)\cdots (\beta-n)(s+\lambda)^{\beta-n-1}$. Thus $(-1)^n\psi^{(n)}(s) \geq 0,\;n=0,1,2,\cdots.$ Similarly, considering $\zeta(s)=(\rho s+\lambda)^{\beta-1}$ we can show that $\zeta(s)$ is completely monotone. Moreover, $(-1)^n\phi^n(s)=(-1)^n\psi^n(s)-(-1)^n\zeta^n(s)$ and $\psi^n(s)$ is always greater than $\zeta^n(s)$ as $0<\rho<1$, this implies $(-1)^n\phi^n(s)\geq 0$. So $\phi(s)$ has completely monotone derivative.
\end{proof}

\noindent Next, we provide the explicit form of the pdf of the error term using complex inversion formula, which generalize the corresponding results in Abraham and Balakrishna \cite{Abraham1999}.
\begin{theorem}
For the stationary tempered stable autoregressive model defined in \eqref{main} with Laplace transform defined as
\begin{equation}\label{tem-LT1}
\mathbb{E}\left(e^{-s X_n}\right)=
e^{-\big((s + \lambda)^{\beta}-\lambda^{\beta}\big)}, \end{equation}
the pdf of the innovation term has following integral form

\begin{align}\label{error_density}
g(x)&=\frac{1}{2\pi i}\int_{0}^{\infty}  \Bigg( e^{-\lambda x/\rho} e^{-wx} e^{\rho^\beta w^\beta e^{\iota \pi \beta}- (we^{\iota \pi}+\lambda-\lambda/\rho)^\beta} 
-e^{-\lambda x/\rho} e^{-wx} e^{\rho^\beta w^\beta e^{-\iota \pi \beta}- (we^{-\iota \pi}+\lambda-\lambda/\rho)^\beta}\Bigg) dw\nonumber\\
&+\frac{1}{2\pi i}\int_{0}^{(\lambda/\rho)-\lambda} \Bigg(e^{-\lambda x} e^{-wx} e^{(-\lambda \rho+\rho we^{\iota\pi}+\lambda)^\beta-(we^{\iota \pi})^\beta}
-e^{-\lambda x} e^{-wx} e^{(-\lambda \rho+\rho we^{-\iota\pi}+\lambda)^\beta-(we^{-\iota \pi})^\beta}\Bigg)dw,
\end{align}
where $\lambda>0$ and $\beta\in (0,1)$.

\end{theorem}

\begin{proof}
Inverse Laplace transform corresponding to $\Phi_\epsilon(s)$ provides the density of innovation term. The density for innovation can be computed using complex inversion formula for Laplace transform given by
\begin{equation}\label{contour}
g(x) = \frac{1}{2\pi i}\int_{x_{0}-i\infty}^{x_{0}+i\infty}e^{sx} \Phi_{\epsilon}(s)ds,
\end{equation}
where point $x_0>a$ for some $a$ is taken in such a way that the integrand is analytic for $\mathcal{R}e(s)>a.$  Note that the integrand $e^{sx} \Phi_{\epsilon}(s)$ is an exponential function which is analytic in whole complex plane. However, due to fractional power in the exponent the integrand $e^{sx} \Phi_{\epsilon}(s)$ has branch points at $s=-\lambda$ and $s=-\lambda/\rho$. Thus we take a branch cut along the non-positive real axis and consider a (single-valued) analytic branch of the integrand. To calculate \eqref{contour}, consider a closed double-key-hole contour $C$: ABCDEFGHIJA (Fig. 1) with branch points at $O = (-\lambda, 0)$ and $O' = (-\lambda/\rho, 0)$. The contour consists of following segments: AB and IJ are arcs of radius $R$, BC, DE, FG and HI are line segment parallel to $x$-axis, CD, EF and GH are arcs of circles with radius $r$ and JA is the line segment from $x_{0}-iy$ to $x_{0}+iy$ (see Fig. 1). The integrand is analytic within and on the contour $C$ so that by Cauchy’s residue theorem (see e.g. Schiff \cite{Schiff1999}) 
\begin{equation}\label{resdue}
\frac{1}{2\pi i}\int_{C}e^{sx}\Phi_{\epsilon}(s)ds = 0.
\end{equation}
\begin{figure}[ht!]
  \centering
    \includegraphics[width=0.8\linewidth]{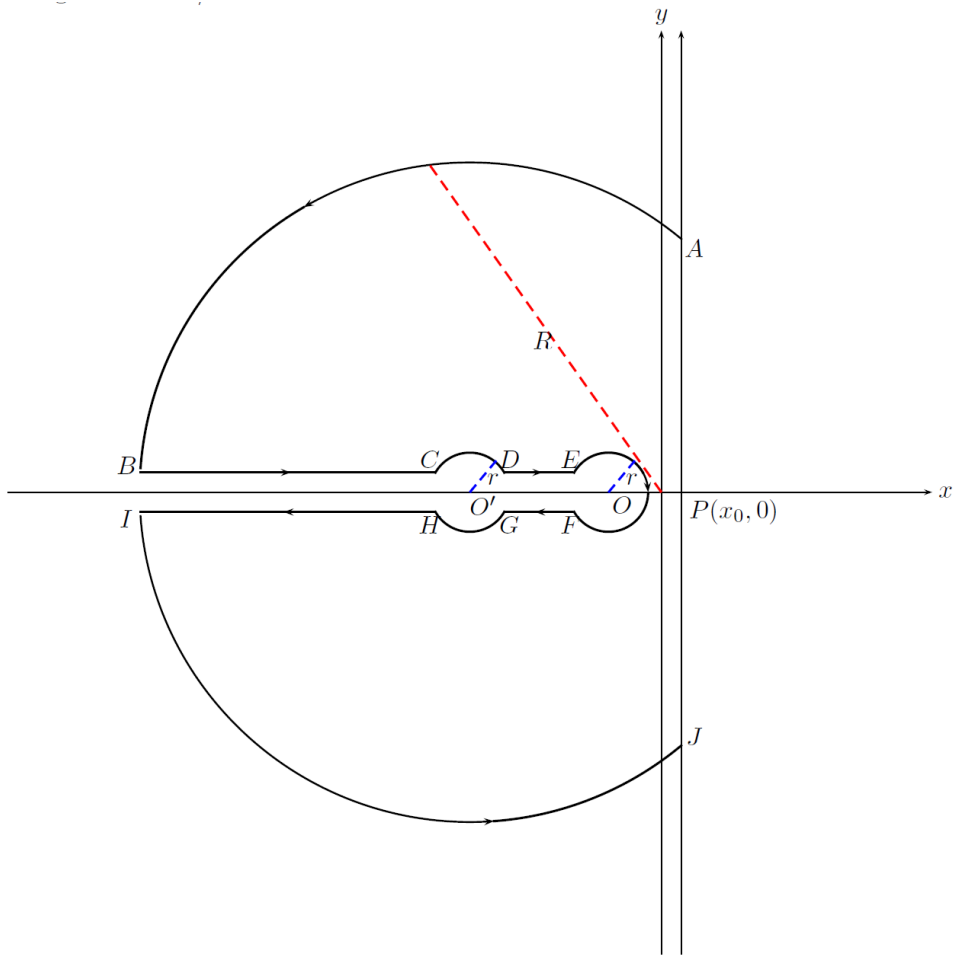}
    \caption{The double key hole contour}
\end{figure}

\noindent Using (Lemma $4.1$ Schiff \cite{Schiff1999}, p. 154), the integral on circular arcs AB and IJ tend to 0 as $R\rightarrow\infty$.  The integral over CD, EF  and GH are also zero as $r$ goes to infinity. Thus, as $r\rightarrow 0$ and $R\rightarrow \infty$, we have

\begin{align}\label{1}
\frac{1}{2\pi i}\int_{x_{0}-i\infty}^{x_{0}+i\infty}e^{sx} \Phi_{\epsilon}(s) &= -\frac{1}{2\pi i}\int_{BC}e^{sx}\Phi_{\epsilon}(s)ds - \frac{1}{2\pi i}\int_{DE}e^{sx}\Phi_{\epsilon}(s)ds\nonumber\\
&\hspace{1cm}-\frac{1}{2\pi i}\int_{FG}e^{sx}\Phi_{\epsilon}(s)ds-\int_{HI}e^{sx}\Phi_{\epsilon}(s)ds.
\end{align}

\noindent Along $BC$, put $s=-\dfrac{\lambda}{\rho}+w e^{\iota \pi}$, which implies $ds=-dw$, and hence
\begin{equation}\label{2}
\int_{BC}e^{sx}\exp((\rho s+\lambda)^\beta-(s+\lambda)^\beta))ds= \int_{r}^{R-\lambda/\rho}   e^{-\lambda x/\rho} e^{-wx} e^{\rho^\beta w^\beta e^{\iota \pi \beta}- (we^{\iota \pi}+\lambda-\lambda/\rho)^\beta} dw.
\end{equation} 

\noindent Along $HI$, $s=-\dfrac{\lambda}{\rho}+w e^{-\iota \pi}$ and $ds=-dw$
\begin{equation}\label{3}
\int_{HI}e^{sx}\exp((\rho s+\lambda)^\beta-(s+\lambda)^\beta))ds=- \int_{r}^{R-\lambda/\rho}   e^{-\lambda x/\rho} e^{-wx} e^{\rho^\beta w^\beta e^{-\iota \pi \beta}- (we^{-\iota \pi}+\lambda-\lambda/\rho)^\beta} dw.
\end{equation} 

\noindent Along $DE$, $s=-\lambda+w e^{\iota \pi}$ and $ds=-dw$
\begin{equation}\label{4}
\int_{DE}e^{sx}\exp((\rho s+\lambda)^\beta-(s+\lambda)^\beta))ds= \int_{r}^{-r+(\lambda/\rho)-\lambda} e^{-\lambda x} e^{-wx} e^{(-\lambda \rho+\rho we^{\iota\pi}+\lambda)^\beta-(we^{\iota \pi})^\beta} dw.
\end{equation}

\noindent Along $FG$, \hspace{2mm}$s=-\lambda+w e^{-\iota \pi}$ and $ds=-dw$
\begin{equation}\label{5}
\int_{DE}e^{sx}\exp((\rho s+\lambda)^\beta-(s+\lambda)^\beta))ds=- \int_{r}^{-r+(\lambda/\rho)-\lambda}   e^{-\lambda x} e^{-wx} e^{(-\lambda \rho+\rho we^{-\iota\pi}+\lambda)^\beta-(we^{-\iota \pi})^\beta} dw
\end{equation}
\noindent Substituting \eqref{2}, \eqref{3}, \eqref{4} and \eqref{5} in  \eqref{1}, the result follows as $r\rightarrow 0$ and $R\rightarrow \infty.$ 
\end{proof}

\begin{remark} For the special case $\lambda=0$ and $\beta=\dfrac{1}{2}$, \eqref{error_density} reduces to
\begin{equation}
g(x)= \Bigg(\dfrac{1-\sqrt{\rho}}{\sqrt 2}\Bigg) \Bigg(\dfrac{1}{2\pi x^3}\Bigg)^{1/2} \exp\Bigg(\frac{(1-\sqrt{\rho})^2}{4x}\Bigg),\hspace{2mm} |\rho| <1,
\end{equation}
and is called L\'evy probability density function see Applebaum \cite{Applebaum2004}, which is a special case of the inverse Gaussian density. Abraham and Balakrishna \cite{Abraham1999} show that if the inverse Gaussian autoregressive model is stationary the error term is also inverse Gaussian distributed. Our results complement their findings. 
\end{remark}
\begin{remark}
For $\lambda =0,$ the error density function in \eqref{error_density} reduces to
\begin{equation}\label{tail_index}
g(x) = \frac{1}{\pi}\int_{0}^{\infty}e^{-wx}e^{w^{\beta}(1-\rho^{\beta})\cos(\pi\beta)}\sin\left(w^{\beta}(1-\rho^{\beta})\sin(\pi\beta)\right),\; x>0.
\end{equation}
\end{remark}

\begin{proposition}\label{2.2}
For $\lambda=0$ the innovation term has the following fractional $qth$ order moment 
\begin{equation}
\mathbb{E}(\epsilon^q)=\frac{\Gamma(1-q/\beta)}{\Gamma(1-q)}(1-\rho^{\beta})^{q/\beta}, \; 0<q<\beta<1.
\end{equation}
\end{proposition}

\begin{proof}
For a positive random variable $X$ with a Laplace transform $\phi(s)$, and for $p>0$, it follows
\begin{align*}
\int_{0}^{\infty} \dfrac{d^n}{ds^n}[\phi(s)] s^{p-1}&= \int_{0}^{\infty} \dfrac{d^n}{ds^n}\left[\mathbb{E}e^{-sX}\right] s^{p-1} ds =\mathbb{E}\int_{0}^{\infty} \dfrac{d^n}{ds^n}\left[e^{-sX}\right] s^{p-1} ds\\
&=(-1)^n\mathbb{E}\Bigg[\int_{0}^{\infty} X^n e^{-uX} u^{p-1} ds\Bigg] =(-1)^n \Gamma(p)\mathbb{E}(X^{n-p}).
\end{align*}
Thus for $q \in (n-1,n)$, for an integer $n$, we have
\begin{align*}
\mathbb{E}(X^q)&=\dfrac{(-1)^n}{\Gamma(n-q)}\int_{0}^{\infty} \dfrac{d^n}{ds^n}\left[\phi(s)\right] s^{n-q-1} ds.
\end{align*}
\noindent Using the relationship for $q \in (0,1)$, it follows
\begin{align*}
\mathbb{E}(\epsilon^q)&=\dfrac{-1}{\Gamma(1-q)}\int_{0}^{\infty} \dfrac{d}{ds}\left[e^{(\rho s+\lambda)^\beta-(s+\lambda)^\beta)}\right] s^{-q} ds\\
&=\dfrac{1}{\Gamma(1-q)}\int_{0}^{\infty} e^{(\rho s+\lambda)^\beta-(s+\lambda)^\beta)}\left((s+\lambda)^{\beta-1}-\rho\beta(\rho s+\lambda)^{\beta-1}\right) s^{-q} ds,
\end{align*}
which gives the $q$th order fractional moment of the innovation term.
\noindent For $\lambda=0$, it can be written in the explicit form as follows
\begin{align*}
\mathbb{E}(\epsilon^q)&=\dfrac{1}{\Gamma(1-q)}\int_{0}^{\infty} e^{-(1-\rho^\beta) s^\beta} s^{\beta-q-1} \beta(1-\rho^\beta) ds\\
&= \dfrac{1}{\Gamma(1-q)}\int_{0}^{\infty} e^{-(1-\rho^\beta) u} u^{-q/\beta}(1-\rho^\beta)  du\;\;\; \mbox{(put $s^{\beta} =u$)}\\
&=\dfrac{(1-\rho^\beta)}{\Gamma(1-q)}\dfrac{\Gamma(1-q/\beta)}{(1-\rho^\beta)^{(1-q/\beta)}} = \frac{\Gamma(1-q/\beta)}{\Gamma(1-q)}(1-\rho^{\beta})^{q/\beta},\;0<q<\beta.
\end{align*}
\end{proof}

\begin{proposition}
For the stationary TAR(1) model defined in \eqref{main} the autocorrelation for $r$th lag $\rho_r$ has the following recursive form
\begin{equation}
    \rho_r=\rho^r.
\end{equation}
\end{proposition}

\begin{proof}
For the given TAR(1) process $X_n=\rho X_{n-1}+\epsilon_n$, multiplying both the side by $X_{n-r}$, we get
\begin{equation}
    X_n X_{n-r}=\rho X_{n-1}X_{n-r}+\epsilon_n X_{n-r}.
\end{equation}
Here $X_n$ is independent of $\epsilon_j$ for $j>n$. Taking expectation on both the sides, yields

\begin{align*}
    \mathbb{E}( X_n X_{n-r})&=\rho\mathbb{E}(X_{n-1}X_{n-r})+\mathbb{E}(\epsilon_n X_{n-r}),
\end{align*}
which gives
\begin{align*}
    \mathbb{E}( X_n X_{n-r})-\mathbb{E}(X_n)\mathbb{E}(X_{n-r})&=\rho\mathbb{E}( X_{n-1}X_{n-r})-\rho\mathbb{E}(X_{n-1})\mathbb{E}(X_{n-r})+\rho\mathbb{E}(X_{n-1})\mathbb{E}(X_{n-r})\\
    &-\mathbb{E}(X_n)\mathbb{E}(X_{n-r})+\mathbb{E}(\epsilon_n)\mathbb{E}(X_{n-r}).
\end{align*}
Thus,
\begin{align*}
\mathrm{Cov}(X_n, X_{n-r})&=\rho\mathrm{Cov}( X_{n-1}X_{n-r})+\mathbb{E}(X_{n-1})\mathbb{E}(X_{n-r})(\rho-1)+\mathbb{E}(\epsilon_n)\mathbb{E}(X_{n-r})\\
&= \rho\mathrm{Cov}( X_{n-1}X_{n-r})+\mathbb{E}(X_{n-r})(\mathbb{E}(X_n)(\rho-1)+\mathbb{E}(\epsilon_n)).
\end{align*}
\noindent Since $\mathbb{E}(\epsilon_n)=(1-\rho)\beta\lambda^{\beta-1}$ and $\mathbb{E}(X_n)=\beta\lambda^{\beta-1}$, it follows
\begin{align*}
    \mathrm{Cov}(X_n, X_{n-r})&=\rho\mathrm{Cov}(X_{n-1},X_{n-r})+\mathbb{E}(X_{n-r})(\beta\lambda^{\beta-1}(\rho-1)+(1-\rho)\beta\lambda^{\beta-1})\\
    &=\rho\mathrm{Cov}(X_{n-1},X_{n-r}).
\end{align*}
\noindent Dividing both side by ${\sqrt{\mathrm{Var}(X_{n}){\rm Var}(X_{n-r})}}$, yields
\begin{align*}
\dfrac{\mathrm{Cov}(X_nX_{n-r})}{\sqrt{\mathrm{Var}(X_n)Var(X_{n-r})}}=\dfrac{\mathrm{Cov}(X_{n-1}X_{n-r})}{\sqrt{\mathrm{Var}(X_{n-1}){\rm Var}(X_{n-r})}},
\end{align*}
Using, $\mathrm{Var(X_n)}$=$\mathrm{Var(X_{n-1})}$=$\mathrm{Var(X_{n-r})}$, leads to
\begin{equation*}
\rho_r=\rho\rho_{r-1} \implies \rho_r=\rho^r,
\end{equation*}
since $\rho_0=1$ and $\rho_1=\rho$.

\end{proof}

\subsection{Parameter Estimation}
Here, we discuss the estimation procedure of the parameters of the tempered stable autoregressive model. We consider conditional least square and moments based estimation methods.\\

\noindent {\bf Parameter estimation by conditional least square:} The conditional least square method provides a straightforward procedure to estimate the parameters for dependent observations by minimizing the sum of square of deviations about conditional expectation, see e.g. Klimko and Nelson \cite{Klimko1978}. To estimate the parameter $\rho$ of process \eqref{main} the conditional least square (CLS) method is used. The CLS estimator of parameter vector $\theta=(\rho, \lambda, \beta)$ is obtained by minimizing 

 \begin{equation}\label{CLS1}
 Q_n(\theta)=\sum_{i=1}^{n}\left[X_{i+1}-\mathbb{E}(X_{i+1}|X_i,X_{i-1},\cdots,X_1)\right]^2,
 \end{equation}
 with respect to $\theta$. For an AR(1) sequence due to Markovian property
\begin{equation}
\mathbb{E}(X_{i+1}|X_i,X_{i-1},\cdots,X_1)=\mathbb{E}(X_{i+1}|X_i),
\end{equation}
which yields
\begin{equation}\label{CLS2}
\mathbb{E}(X_{i+1}|X_i)=\rho X_i+\mathbb{E}(\epsilon_i).
\end{equation}

\noindent To calculate the first moment of error term, taking $\left.\dfrac{d\epsilon_i}{ds}\right\vert_{s=0}$, gives
\begin{equation}\label{CLS3}
	\mathbb{E}(\epsilon_i)=\beta \lambda^{\beta-1}(1-\rho).
\end{equation}
Using the equations \eqref{CLS1}, \eqref{CLS2} and \eqref{CLS3} we can write
\begin{equation}\label{min}
 Q_n(\theta)= \sum_{i=1}^{n}\left[X_{i+1}-\rho X_i-(\rho-1)\beta\lambda^{\beta-1}\right]^2.
\end{equation}

\noindent Minimising \eqref{min} w.r.t $\rho, \beta,$ and $\lambda$ gives the following estimate for $\rho$ and an estimating equation for $\beta$ and $\lambda$.

\begin{equation}\label{rho_estimate}
\hat{\rho}=\dfrac {\sum_{i=1}^{n}(X_i-\bar{X}) (X_{i+1}-\bar{X})}{\sum_{i=1}^{n}(X_i-\bar{X})^2},
\end{equation}
and
\begin{equation}\label{Eq_beta_lambda}
\beta \lambda^{\beta-1}=\dfrac{\sum_{i=1}^{n}(X_{i+1}-\hat{\rho}X_i)}{n(1-\hat{\rho})}.
\end{equation}
The estimate $\hat{\rho}$ can be easily calculated from the observed sample using \eqref{rho_estimate}. However, some numerical method which minimizes the squared difference between left and right hand sides is required to estimate $\beta$ and $\lambda$ from \eqref{Eq_beta_lambda}.\\

\noindent{\bf Moments based estimation:} The parameters of positive tempered stable distribution cannot be estimated by classical maximum likelihood estimation method due to non availability of closed from probability density function. Also the CLS estimator provides an estimate for $\rho$ only. Here we use method of moments to estimate the parameters $\beta$ and $\lambda$. The general idea used by Pearson \cite{Pearson1894} is, the moments of the PDF about the origin are equal to the corresponding moments of the sample data. In case of TAR(1) process moments for innovation distribution can be extracted using the corresponding Laplace transform as the density \eqref{error_density} is in complex form. The idea is to apply the method of moments on the error sample which can be easily obtained from the original series by first estimating the $\rho$ using CLS method. The $k$th order moment can be extracted from taking $k$th derivative of Laplace transform as follows for the innovation term.

\begin{align*}
\mathbb{E}[\epsilon^k]=(-1)^k \dfrac{d^k}{ds^k}\phi_{\epsilon}(s),\;k \in \mathbb{N}.
\end{align*}
\vspace{0.5mm}

\noindent Using the Laplace transform defined in \eqref{der}, we have
\begin{align}\label{pop_mean}
\mathbb{E}(\epsilon_i)=\left.\dfrac{d}{ds}\phi_{\epsilon}(s)\right\vert_{s=0}&=\beta\lambda^{\beta-1}(1-\rho).
\end{align}
\begin{align}\label{pop_mom2}
\mathbb{E}(\epsilon_i^2)=\left.\dfrac{d^2}{ds^2}\phi_{\epsilon}(s)\right\vert_{s=0}&=(\beta\lambda^{\beta-1}(\rho-1))^2+\beta(\beta-1)\lambda^{\beta-2}(\rho^2-1).
\end{align}


\noindent The sample estimates for $\mathbb{E}(\epsilon_i)$ and $\mathbb{E}(\epsilon_i^2)$ are given by
\begin{equation}\label{mom3}
m_1=\dfrac{\sum_{i=1}^{n}\epsilon_i}{n}=\dfrac{\sum_{i=1}^{n}(X_{i+1}-\rho X_i)}{n}
\end{equation}

\begin{equation}\label{mom4}
m_2=\dfrac{\sum_{i=1}^{n}\epsilon_i^2}{n}=\dfrac{\sum_{i=1}^{n}(X_{i+1}-\rho X_i)^2}{n}.
\end{equation}

\vspace{2mm}
\noindent Letting, $z=\dfrac{\sum_{i=1}^{n}(X_{i+1}-\rho X_i)}{n}$ and $r=\dfrac{\sum_{i=1}^{n}(X_{i+1}-\rho X_i)^2}{n}$ and using \eqref{pop_mean}, \eqref{pop_mom2}, \eqref{mom3} and \eqref{mom4} we get the following relations

\begin{align}
\beta\lambda^{\beta-1}=\dfrac{z}{1-\rho},\;\;\; \beta =1+\dfrac{(z^2-r)\lambda}{z(1+\rho)}.
\end{align}

\noindent After some manipulation, the above two equations leads to following nonlinear equations

\begin{align*}
\beta\Bigg(\dfrac{\beta-1}{d}\Bigg)^{(\beta-1)}-\Bigg(\dfrac{z}{1-\rho}\Bigg)&=0\\
(1+d\lambda)\lambda^{d\lambda}-\Bigg(\dfrac{z}{1-\rho}\Bigg)=0
\end{align*}

\noindent We use sequential least square optimization technique (SLSQP) to solve these equations which is is an iterative method to optimize a nonlinear model using the nonlinear least square model and sequential quadratic programming model (see Fu et al. \cite{Fu2019}). We use scipy.optimize which use SLSQP as an inbuilt method in python to solve nonlinear equations that gives the estimates for $\beta$ and $\lambda$. 
\vspace{2mm}

\noindent {\bf Parameter estimation by tail index:} The parameter estimation for the case $\lambda=0$ (corresponding density given by \eqref{tail_index}) is done by using the right tail behaviour of $S$ for the TAR(1) process. For the TAR(1) process we have a  positive stable random variable $(1-\rho^{\beta})^{1/\beta}S$ with corresponding Laplace transform $e^{-(1-\rho^{\beta})s^{\beta}}$. Using \eqref{tail-stable}, it follows 
\begin{align}\label{tail_error}
    \mathbb{P}((1-\rho^\beta)^{1/\beta}S>x)&=\mathbb{P}(S> (1-\rho^\beta)^{-1/\beta}x)\sim \frac{ ((1-\rho^\beta)^{-1/\beta} x)^{-\beta}}{\Gamma{(1-\beta)}},\;\mbox{as}\;x\rightarrow\infty.
\end{align}

\noindent Taking $\log$ on both sides of \eqref{tail_error} and replacing $\mathbb{P}(S> (1-\rho^\beta)^{-1/\beta}x)= G_\epsilon(x)$, where $G_\epsilon(x)$ represents the empirical tail for $\epsilon$, we get
\begin{equation}\label{tail_stable}
    \log(G_\epsilon(x))=-\beta\log(x)+\log(1-\rho^\beta)-\log(\Gamma(1-\beta)),\;0<\beta<1.
\end{equation}

\noindent Considering $Y=\log(G_\epsilon(x))$ and $X=-\log(x)$, which yields $\beta$ as slope of the regressing $Y$ on $X$.

\subsection{Simulation Study}
To check the efficacy of the estimation procedure we perform simulation. Here we have considered an AR$(1)$ model with marginals to be distributed as one sided tempered stable distribution and the corresponding Laplace transform for error term $\epsilon_n$ is given in \eqref{der}. In order to simulate a tempered stable AR(1) series, first we simulate data for innovation using the Laplace transform, see e.g. Ridout (2008). Ridout \cite{Ridout2008} discussed the simulation of a continuous distribution function using its Laplace transform if the density of the distribution is not in closed form. 

We consider three different simulated datasets of length $n=10,000$ each with three different parameter combinations. For the first case consider, $\rho=0.9$, $\lambda=1$ and $\beta=0.5$ and the obtained estimates are $\hat{\rho}=0.895$, $\hat{\lambda}=1.13.$ and $\hat{\beta}=0.489$. 
For the second case we consider, $\rho=0.8$, $\lambda=2$ and $\beta=0.7$ and the obtained estimates are $\hat{\rho}=0.795$, $\hat{\lambda}=1.93$ and $\hat{\beta}=0.709$.
At last $\rho=0.75$, $\lambda=1.5$ and $\beta=0.9$ are taken and the obtained estimates are $\hat{\rho}=0.755$, $\hat{\lambda}=1.601$ and $\hat{\beta}=0.90$. The results are summarized in Table 1.

\begin{table}[ht!]
\begin{center}
\begin{tabular}{||c||c||c|}
\hline 
 & Actual & Estimated \\ 
\hline 
Case 1 & $\rho=0.9$, $\lambda=1$, $\beta=0.5$&$\hat{\rho}=0.895$, $\hat{\lambda}=1.13$, $\hat{\beta}=0.489$\\ 
\hline 
Case 2&  $\rho=0.8$, $\lambda=2$, $\beta=0.7$ & $\hat{\rho}=0.795$, $\hat{\lambda}=1.93$, $\hat{\beta}=0.70$\\ 
\hline
Case 3&  $\rho=0.75$, $\lambda=1.5$, $\beta=0.9$ &  $\hat{\rho}=0.755$, $\hat{\lambda}=1.60$, $\hat{\beta}=0.90$\\ 
\hline 
\end{tabular}
\caption{Actual and estimated parameters values for different choices of parameters}
\end{center}
\end{table}

 \noindent The sample trajectory for an AR$(1)$ process for positive tempered stable distribution with parameters $\rho=0.9$, $\lambda=2$ and for $\beta=0.5$ is given in Fig. 2.
\begin{figure}[ht!]%
    \centering
    {\includegraphics[width=16cm, height=7cm]{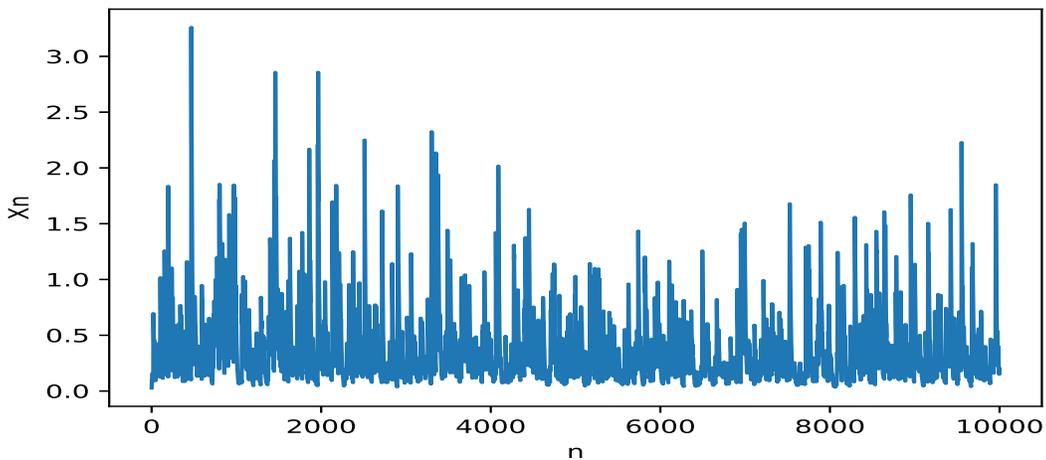} }%
    \caption{Sample path of TAR(1) process}%
\end{figure}

\noindent Next we use resampling method to generate several samples of the simulated data. To ensure the stationarity here we assume AR$(1)$ process with parameter $\rho < |1|$, as the AR$(1)$ process is stationary if $|\rho|< 1$. Consider $\rho=0.9$ and take tempered stable distribution parameters as $\beta=0.5$ and $\lambda=2$, by iterative resampling of simulated data we generate $1000$ simulated samples for the defined parameters. The parameters are estimated using method of moments for $\lambda$, $\beta$ and based on conditional least square for $\rho$ considering different samples. The corresponding box-plots for estimated parameter values are shown in Fig. 3.

\begin{figure}[ht!]%
    \centering
    {\includegraphics[width=14cm, height =7cm]{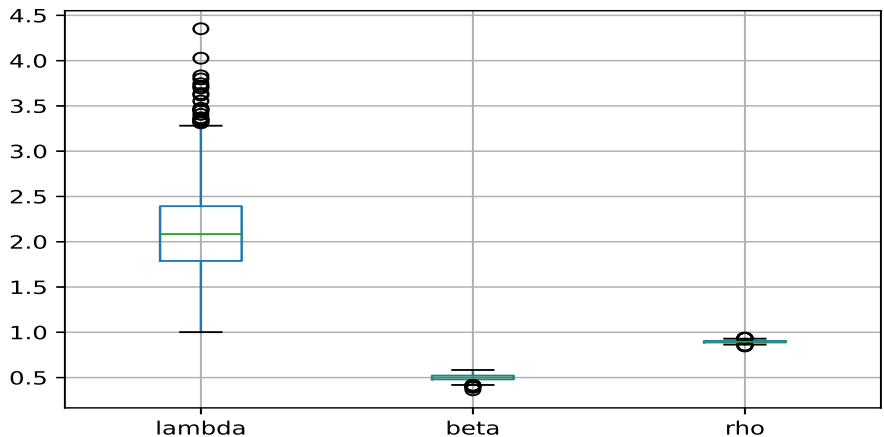} }%
    \caption{Box-plot of estimated parameters of TAR process using method of moments}%
\end{figure}

\noindent For the particular case, $\lambda=0$ which leads to one-sided positive stable autoregressive model, simulation is done by the method given in Ridout \cite{Ridout2008} and parameter estimation is done by the tail index method using \eqref{tail_stable}. We simulate a sample of $10,000$ observations for $\rho=0.5$ and $\beta=0.5$. By taking values greater than a particular threshold such that we have some percentage of the sorted dataset greater than that threshold value, we get regression based estimator for $\beta$. Taking a cutoff value of $0.009$ gives the $30\%$ of observation greater than the cutoff and gives the corresponding estimate of $0.36$ for $\beta$. For cutoff value $0.019$, we get almost $20\%$ of the observation greater than threshold value and corresponding estimate is $0.41$ for $\beta$ which is more close to true parameter value. Further, for $k=0.06$ the estimated $\beta$ is 0.49 which is even more closer to the actual value. This clearly indicates that the accuracy of the estimated parameter depend on the threshold selected. The results are summarized in Table 2, where $k$ represents the threshold value.

\begin{table}[ht!]
\begin{center}
\begin{tabular}{||c||c||c||c|}
\hline 
 & Threshold & Actual $\beta$ & Estimated $\beta$ \\ 
\hline 
Case 1 & k=0.009 & 0.5 & 0.36\\ 
\hline 
Case 2& k=0.019 & 0.5 & 0.41\\
\hline
Case 3& k=0.06 & 0.5 & 0.49 \\
\hline 
\end{tabular}
\caption{Estimated parameter values based on tail index for $\lambda=0$}
\end{center}
\end{table}

\section{AR(1) Process with Tempered Stable Innovations}
Consider the stationary AR(1) process
\begin{equation}\label{eqn1}
X_n=\rho X_{n-1}+\epsilon_n,\;\;|\rho|<1,
\end{equation}
\noindent where $\epsilon_n\textsc{\char13}s$ are i.i.d positive tempered stable random variables defined in \eqref{tem-LT} and independent of $X_{n-1}$. Let $\phi_{X_n}(s)$ and $\phi_{\epsilon_n}(s)$ are Laplace transform of $X_n$ and $\epsilon_n$ respectively. Now assuming $X_0\overset{\mathrm{d}}{=}\epsilon_0$ and recursively writing equation $\eqref{eqn1}$, we get a moving average representation of AR(1) process as follows

\begin{equation}
X_n=\sum_{i=0}^{n}\rho^i\epsilon_{n-i},\; \lambda>0, \;\beta\in (0,1).
\end{equation}
The Laplace transform of $X_n$ is
\begin{align}\label{eqn2}
\phi_{X_n}(s)=\prod_{i=0}^{n}\phi_{\epsilon}(\rho^is) = \prod_{i=0}^{n}e^{\lambda^\beta-(\lambda+\rho^i s)^{\beta}}
=e^{n\lambda^\beta-\sum_{i=0}^{n}(\lambda+\rho^i s)^{\beta}}.
\end{align}




\begin{proposition}
For $\lambda=0$, the $qth$ order moment of $X_n$ is given by
\begin{equation}
\mathbb{E}(X_n^q)=\frac{\Gamma(1-q/\beta)}{\Gamma(1-q)}(\sum_{i=0}^{n}\rho^{i\beta})^{q/\beta}, \; 0<q<\beta<1.
\end{equation}

\end{proposition}

\begin{proof}
Using the proposition \eqref{2.2}, for a positive random variable $X$ with corresponding Laplace transform $\phi(s)$, and for $p>0$, it follows
\begin{align*}
\int_{0}^{\infty} \dfrac{d^m}{ds^m}[\phi(s)] s^{p-1}=(-1)^n \Gamma(p)\mathbb{E}(X^{n-p}).
\end{align*}
Thus for $q \in (m-1,m)$, for an integer $m$, we have
\begin{align*}
\mathbb{E}(X^q)&=\dfrac{(-1)^m}{\Gamma(n-q)}\int_{0}^{\infty} \dfrac{d^m}{ds^m}\left[\phi(s)\right] s^{n-q-1} ds.
\end{align*}
\noindent For $q \in (0,1)$, it follows
\begin{align*}
\mathbb{E}(X_n^q)&=\dfrac{-1}{\Gamma(1-q)}\int_{0}^{\infty} \dfrac{d}{ds}\left[e^{n\lambda^\beta-\sum_{i=0}^{n}(\rho^is+\lambda)^\beta}\right] s^{-q} ds\\
&=\dfrac{1}{\Gamma(1-q)}\int_{0}^{\infty} e^{n\lambda^\beta-\sum_{i=0}^{n}(\rho^is+\lambda)^\beta}\sum_{i=0}^{n}\beta \rho^i(\rho^is+\lambda)^{\beta-1} s^{-q} ds.
\end{align*}
\noindent For $\lambda=0$, it yields
\begin{align*}
\mathbb{E}(X_n^q)&=\dfrac{1}{\Gamma(1-q)}\int_{0}^{\infty} e^{\sum_{i=0}^{n}\rho^{i\beta} s^\beta} s^{\beta-q-1} \beta(\sum_{i=0}^{n}\rho^{i\beta}) ds\\
&= \dfrac{1}{\Gamma(1-q)}\int_{0}^{\infty} e^{-\sum_{i=0}^{n}\rho^{i\beta}u} u^{-q/\beta}\beta(\sum_{i=0}^{n}\rho^{i\beta})  du\;\;\; \mbox{(put $s^{\beta} =u$)}\\
&=\dfrac{\sum_{i=0}^{n}\rho^{i\beta} }{\Gamma(1-q)}\dfrac{\Gamma(1-q/\beta)}{(\sum_{i=0}^{n}\rho^{i\beta})^{(1-q/\beta)}} = \frac{\Gamma(1-q/\beta)}{\Gamma(1-q)}(\sum_{i=0}^{n}\rho^{i\beta})^{q/\beta},\;0<q<\beta.
\end{align*}
\end{proof}

\begin{remark}
Using same argument as in Prop. \ref{2.1}, we show that the marginal distributions of $X_n$ are infinitely divisible. It is easy to show that for $e^{n\lambda^\beta-\sum_{i=0}^{n}(\lambda+\rho^i s)^{\beta}}$, the exponent $\phi(s)= \sum_{i=0}^{n}(\lambda+\rho^i s)^{\beta} - n\lambda^\beta$ has completely monotone derivative, that is $(-1)^n \phi^n(s)\geq 0,\;n=0,1,2,\cdots.$
\end{remark}

\begin{remark}
The index of dispersion for $X_n$ is  $ID(X_n)=\dfrac{\mathrm{Var}(X_n)}{\mathbb{E}(X_n)}$. Using the mean and variance of the process $X_n$, given by $\mathbb{E}(X_n)=\dfrac{\beta \lambda^{\beta-1}}{1-\rho}$ and $ \mathrm{Var}(X_n)=\dfrac{\beta (1-\beta)\lambda^{(\beta-2)}}{1-\rho^2}$, yields $ID(X_n)=\dfrac{1-\beta}{\lambda(1+\rho)}$, where $0<\rho<1, \beta \in (0,1)$ and $\lambda>0$. For $\lambda>1$ the $ID(X_n)<1$, which implies that the marginal distributions of $X_n$ are under-dispersed for $\lambda>1$.
\end{remark}

\subsection{Parameter Estimation}
To estimate the parameter $\rho$ of the process $\eqref{eqn1}$, the  conditional least square method discussed in previous section is used, which minimize 
\begin{equation}
 Q_n(\theta)=\sum_{i=1}^{n}\left[X_{i+1}-\mathbb{E}(X_{i+1}|X_i,X_{i-1},\cdots,X_1)\right]^2,
 \end{equation}
with respect to $\theta$. Using the equations \eqref{CLS1}, \eqref{CLS2} and mean of $\epsilon_n$, we can write

\begin{equation}\label{min1}
 Q_n(\theta)= \sum_{i=1}^{n}\left[X_{i+1}-\rho X_i-\beta\lambda^{\beta-1}\right]^2.
\end{equation}

\noindent Minimizing \eqref{min1} w.r.t $\rho$ gives the following estimate
\begin{equation}\label{rho_estimate1}
\hat{\rho}=\dfrac {\sum_{i=1}^{n}(X_i-\bar{X}) (X_{i+1}-\bar{X})}{\sum_{i=1}^{n}(X_i-\bar{X})^2}.
\end{equation}

\noindent The parameter estimation for $\beta$ and $\lambda$ is done again using method of moments for $\epsilon_n$, which leads to
\begin{align}
m_1=\beta\lambda^{\beta-1},\;\;m_2=(\beta\lambda^{\beta-1})^2+\beta(1-\beta)\lambda^{\beta-2},
\end{align}
where $m_1$ and $m_2$ are first and second order sample moments for $\epsilon_n$ respectively. After some manipulation, we get the following nonlinear equations
\begin{align}
\beta{(c(\beta-1))}^{(\beta-1)}-m_1=0,\;\;\;
\Bigg(\dfrac{\lambda}{c}+1\Bigg)\lambda^{(\lambda/c)+1}-m_1=0,
\end{align}
\noindent where $c=\dfrac{m_1}{m_2^2-m1}$. Using the optimization technique SLSQP, discussed in Section 3, these equations are solved in python to obtain the estimates $\hat{\lambda}$ and $\hat{\beta}$. The estimate of $\lambda$ can also be obtained using the estimate $\hat{\beta}$, denoted by
\begin{equation}
\hat{\lambda}=\Bigg(\dfrac{m_1}{\hat{\beta}}\Bigg)^{\dfrac{1}{\hat{\beta}-1}}-m_1=0
\end{equation}

\subsection{Simulation Study and Real Data Application}
To check the performance of the estimation method, we use simulated dataset. Using the method discussed in Ridout \cite{Ridout2008}, we simulate an independent sample of size 10,000 for error term using the Laplace transform of positive tempered stable random variable defined in \eqref{tem-LT}. Then we recursively simulate AR(1) series by using fixed value of $\rho$ and letting $\epsilon_0=X_0$. The performance of the estimation method for three different sets of parameters is summarized in Table 3.
The sample trajectory for AR(1) process with tempered stable innovations is given in Fig. 4. Further, to assess the performance of the estimation method, box plots of the estimated parameters $\hat{\rho},\hat{\beta}$ and $\hat{\lambda}$ are plotted by taking multiple simulated random samples of innovations and generating multiple AR(1) time-series. Here we generate $1000$ samples each of length$n=10,000$ for $\rho=0.9, \beta=0.5$ and $\lambda=2$ and estimate parameters for each sample. The box plot of estimated parameters is shown in Fig. 5. 


\begin{table}[ht!]
\begin{center}
\begin{tabular}{||c||c||c|}
\hline 
 & Actual & Estimated \\ 
\hline 
Case 1 & $\rho=0.9$, $\lambda=1$, $\beta=0.5$&$\hat{\rho}=0.8909$, $\hat{\lambda}=1.01$, $\hat{\beta}=0.5002$\\ 
\hline 
Case 2&  $\rho=0.8$, $\lambda=2$, $\beta=0.7$ & $\hat{\rho}=0.804$, $\hat{\lambda}=2.037$, $\hat{\beta}=0.69$\\ 
\hline
Case 3&  $\rho=0.75$, $\lambda=1.5$, $\beta=0.9$ &  $\hat{\rho}=0.75131$, $\hat{\lambda}=1.56$, $\hat{\beta}=0.89$\\ 
\hline 
\end{tabular}
\caption{Actual and estimated parameters values for different choices of parameters}
\end{center}
\end{table}

\begin{figure}[ht!]%
    \centering
    {\includegraphics[width=16cm, height=7cm]{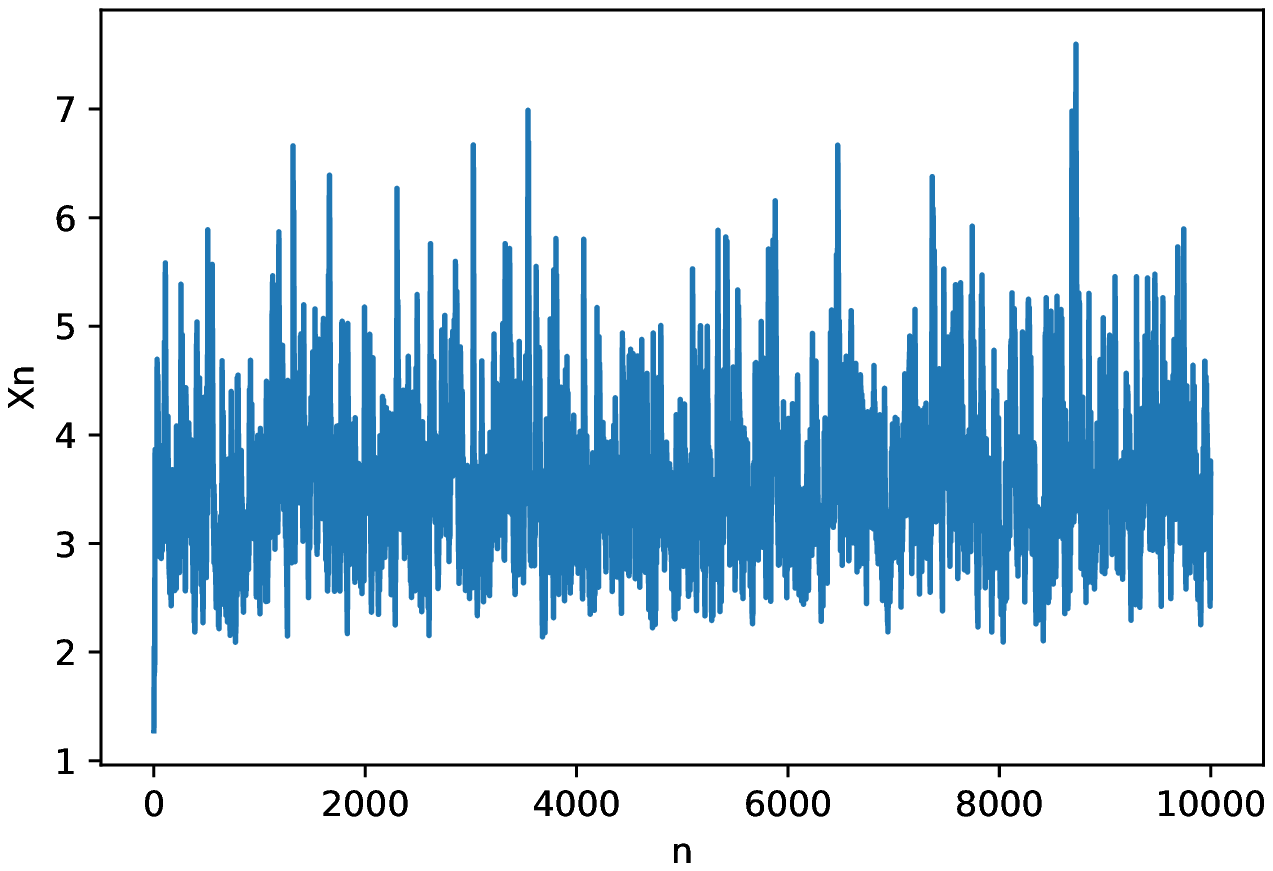} }%
    \caption{Sample path of AR(1) process with tempered stable innovations}%
\end{figure}

\begin{figure}[ht!]%
    \centering
    {\includegraphics[width=16cm, height=7cm]{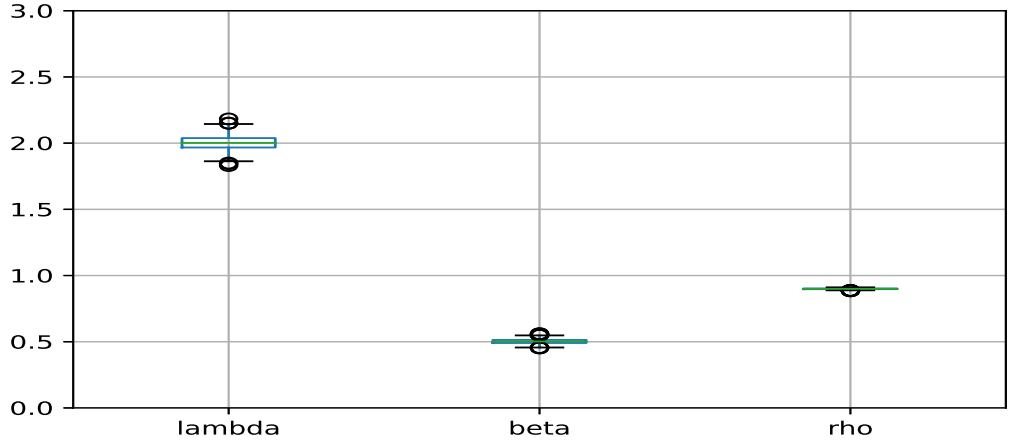} }%
    \caption{Sample path of AR(1) process with tempered stable innovations}%
\end{figure}

\noindent {\bf Real Data Application:} The power consumption data during the time of COVID-19 is extracted from Kaggle. The data contains the state-wise power consumption in India during the time period 2nd Jan 2019 till 23rd May 2020 that is for a period of 17 months and each data-point represents the power consumption at that time in Mega Units (MU). Here we are working on the data during the period 14th April 2019 till 29th April 2020 for the Indian state Arunachal Pradesh. The sample path of the data is plotted in Fig. 6.

\begin{figure}[ht!]%
    \centering
    {\includegraphics[width=14cm, height=7cm]{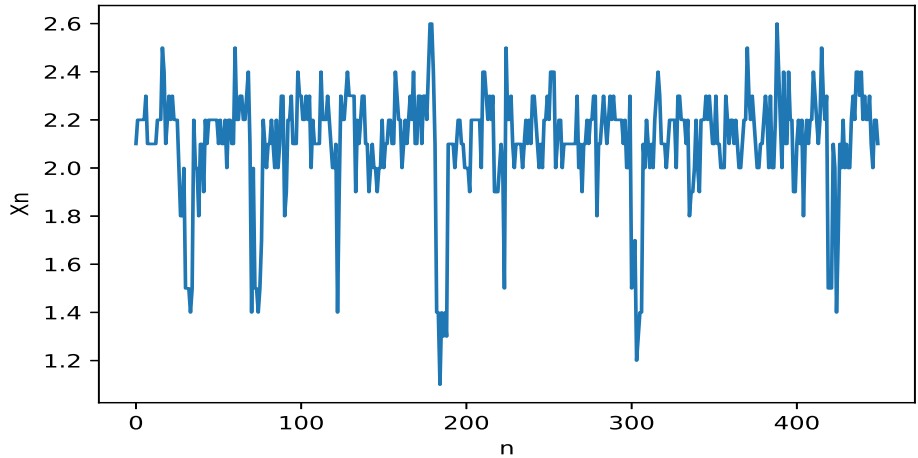} }%
    \caption{The power consumption time-series of Arunachal Pradesh}%
\end{figure}

\noindent By simply looking at the data we can visualize the stationarity with some sharp spikes suggesting non-Gaussian behaviour. We have applied Augmented Dickey–Fuller (ADF) test to check the stationarity of the data and the $p$-value comes out to be less than 0.05 which indicates that the null hypothesis of ADF test is rejected at 5\% significance level and implies that the data is stationary. To determine the appropriate time series model, ACF and PACF plots are given in Fig. 7 and Fig. 8 respectively which determines the appropriate time series model with significant lag. 
\begin{figure}[ht!]%
    \centering
    {\includegraphics[width=16cm, height=7cm]{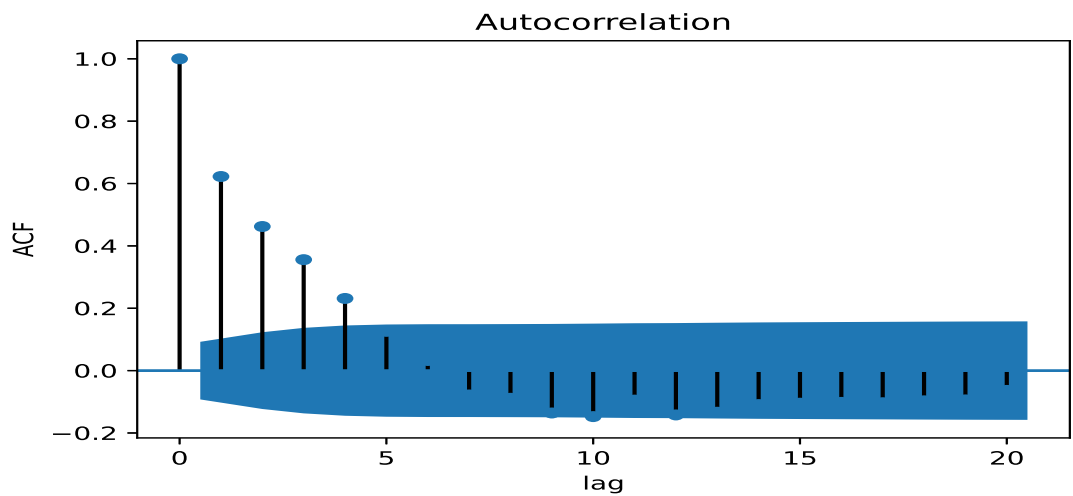} }%
    \caption{ACF plot of data.}%
\end{figure}

\begin{figure}[ht!]%
    \centering
    {\includegraphics[width=16cm, height=7cm]{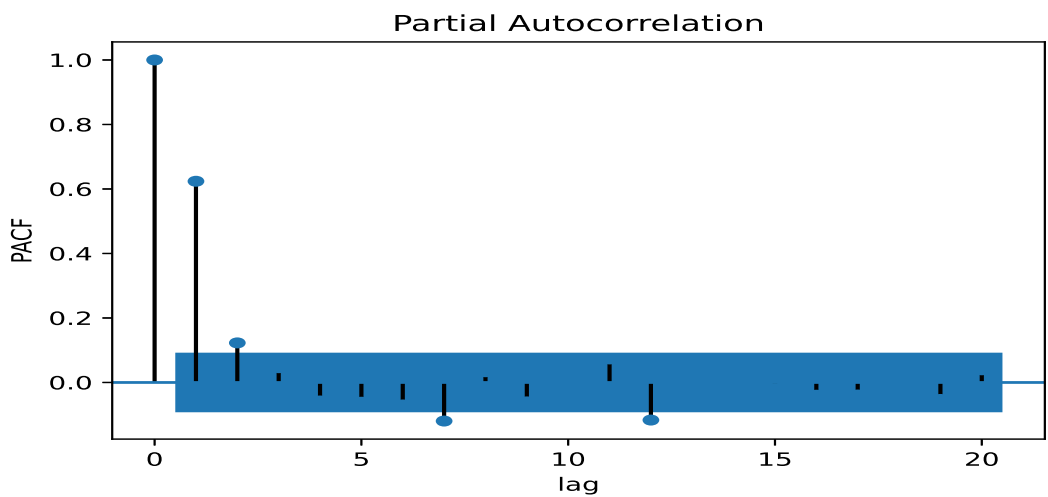} }%
    \caption{PACF plot of data.}%
\end{figure}

\noindent Clearly PACF plot for the dataset is significant till lag $1$ and the ACF plots tails off, this means that AR(1) model would be a good fit for this data. We assume that the innovation term $\epsilon_n$ follows one sided tempered stable distribution. We fit the proposed AR(1) model with tempered stable distribution on the dataset. The estimated parameters using the above described methods are $\hat{\beta}=0.91, \hat{\lambda}=2.9$ and $\hat{\rho}=0.60$. Using these estimated values, we simulate a dataset to assess the accuracy of the results. We simulate a data sample for $\hat{\lambda}=2.8$ and $\hat{\beta}=0.92$. Considering the relation $\epsilon_n=X_n-\rho X_{n-1}$ generates a dataset for the innovation term. To check whether both the simulated data using estimated parameters values and the actual data set follow same distribution or not we use two sample Kolmogorov–Smirnov (K-S) test and Mann–Whitney $U$ test. The two sample K-S and Mann–Whitney $U$ test both are non parametric tests which compare the distribution of two datasets. The two sample K-S test gives a $p$-value 0.056 and Mann–Whitney $U$ test gives the $p$-value 0.377. Thus the null hypothesis for both the tests are accepted at $5\%$ level of significance. This indicates that both the samples are from same distributions. Hence it is appropriate to model the considered power consumption time-series using AR(1) model with positive tempered stable innovations.

\section{Conclusions} We introduced two kinds of tempered stable autoregressive models. The first model denoted by TAR(1) has one-sided tempered stable marginals. The second model discussed in Section 3 has one-sided tempered stable innovations. Autoregressive models with positively distributed innovations are quite common in literature for example inverse Gaussian, Gamma and one-sided stable are used. The autoregressive model with one-sided tempered stable innovations generalizes the inverse Gaussian and positive stable innovations based autoregressive models. We provide the parameter estimation method of the introduced model and assess the efficacy using simulated data. Further, we show the applicability of the introduced model on real life power consumption data. We hope that introduced models will be helpful in modeling of several real life series emerging from different areas like finance, economics, natural hazards, power consumption and others.  \\

\noindent {\bf Acknowledgments:} N.B. would like to thank Ministry of Education (MoE), India for supporting her PhD research. Further, A.K. would like to express his gratitude to Science and Engineering Research Board (SERB), India for the financial support under the MATRICS research grant MTR/2019/000286.
\vone

\end{document}